\documentclass[11pt]{article} 

\usepackage[utf8]{inputenc} 

\usepackage{geometry} 
\usepackage{graphicx} 

\usepackage{booktabs} 
\usepackage{array} 
\usepackage{paralist} 
\usepackage{verbatim} 
\usepackage{subfig} 

\usepackage{amsmath}
\usepackage{amssymb}
\usepackage{amsthm}
\usepackage{url}

\newtheorem{theorem}{Theorem}

\newtheorem{lemma}{Lemma}
\newtheorem{definition}{Definition}

\usepackage{fancyhdr} 
\usepackage{sectsty}
\allsectionsfont{\sffamily\mdseries\upshape} 

\usepackage[nottoc,notlof,notlot]{tocbibind} 
\usepackage[titles,subfigure]{tocloft} 

\title{An algorithmic proof of Bachet's conjecture and the Lagrange-Euler method}
\author{Felix Sidokhine}

\begin{document}
\maketitle

\section{Introduction}

The goal of this notice is to present a proof of Bachet's conjecture based exclusively on the fundamental theorem of arithmetic. The novelty of this proof consists in its introduction of a partial order on rational integers through the unique factorization property. In general, the proofs of Bachet's conjecture by Lagrange - Euler's method (c.f. \cite{Landau1958}, \cite{Davenport1983}, \cite{Cohen2007}) assume necessary the use of infinite descent. In the proposed proof we do not assume the existence of a ``minimal solution", but rather we show the existence of the desired solution through an algorithmic method.
\\
\\
This approach should also be suitable for generalized versions of Bachet's conjecture for algebraic integers. This is due to the fact that total orders are often impossible to introduce in algerbaic extensions of $\mathbb{Q}$. However, if unique factorization is used as the basis for ordering, it is likely to be possible to apply our approach and obtain the desired results. For example, \cite{Deutsch2002} possibly had to restrict her work to totally ordered fields due to the problem of ordering.

\section{Definitions}

\begin{definition}[Initial interval of primes] 
Let $\Pi_n = \{ p_1=2,p_2=3,...,p_n \}$ be the interval of the first $n$ prime numbers.
\end{definition}

\noindent Let $S(\Pi_n) = \{ w = p_1^{i_1}...p_n^{i_n}| i_1,i_2,..,i_n \in \mathbb{Z}^+ \}$. Given an element $w \in S(\Pi_n)$, $w$ must be written as $w= p_1^{\alpha_1}...p_n^{\alpha_n}$, even if some of the powers are 0. The leading prime factor of $w$, is the prime with the greatest index, whose corresponding power is not 0.

\begin{definition}[The $L$ map]
Let $w = p_1^{\alpha_1}p_2^{\alpha_2}...p_n^{\alpha_n}$ where $\alpha_i \in \mathbb{Z}^+$. Let $L :S  \rightarrow \mathbb{Z}^+$,  $L(w) = k$, where $k$ is the index of the leading prime.
\end{definition}

\begin{definition}[The $\nu$ map]
Let $w = p_1^{\alpha_1}p_2^{\alpha_2}...p_n^{\alpha_n}$ where $\alpha_i \in \mathbb{Z}^+$. Let $\nu : S \rightarrow \mathbb{Z}^+$ where $\nu(w) = \alpha_{L(w)}$. 
\end{definition}

\noindent The above mappings are well-defined due to the unique factorization property of $\mathbb{Z}$.

\begin{definition}[The partial order on $S$]
Given $w_1$ and $w_2$, $w_1 \prec w_2$ if $L(w_1) < L(w_2)$ or $L(w_1) = L(w_2)$ and $\nu(w_1) < \nu(w_2)$.
\end{definition}

\noindent We shall now give some properties of this partial order on the set $S$. Let $w_1 = w_2 w_3$, then $L(w_1) = \max(L(w_2),L(w_3))$; and moreover if $L(w_2) < L(w_3)$, then $\nu(w_1) = \nu(w_3)$, otherwise if $L(w_2) = L(w_3)$ then $\nu(w_1) = \nu(w_2) + \nu(w_3)$. 

\begin{definition}[Reduced Solution]
Given a system of equations
\begin{equation*}
\begin{cases}
x_1^2 + x_2^2 + x_3^2 + x_4^2 - px_5 = 0 \\
(x_1,x_2,...,x_5) = 1
\end{cases}
\end{equation*}

\noindent a solution $(a_1,a_2,a_3,a_4,a_5)$ is called a reduced solution if every prime factor of $a_5$ precedes $p$.
\end{definition}

\section{The Result}

\begin{lemma}
For any prime $p$ the system of equations:
\begin{equation}
\begin{cases}
x_1^2 + x_2^2 + x_3^2 + x_4^2 - px_5 = 0 \\
(x_1,x_2,...,x_5) = 1
\end{cases}
\end{equation}
\noindent has a reduced solution.
\end{lemma}

\begin{lemma}
Let $(a_1,a_2,a_3,a_4,a_5)$ be a reduced solution of equation (1) and $p' | a_5$. Then the system of equations:
\begin{equation}
\begin{cases}
y_1^2 + y_2^2 + y_3^2 + y_4^2 - p'y_5 = 0 \\
(y_1,y_2,...,y_5) =1
\end{cases}
\end{equation}

\noindent has a reduced solution $(b_1,b_2,b_3,b_4,b_5)$ such that:
\begin{equation}
\begin{aligned}
a_1b_1 + a_2b_2 + a_3b_3 + a_4b_4 \equiv 0 \mod p' \\
a_1b_2 - a_2b_1 + a_3b_4 - a_4b_3 \equiv 0 \mod p' \\
a_1b_3 - a_3b_1 + a_4b_2 - a_2b_4 \equiv 0 \mod p' \\
a_1b_4 - a_4b_1 + a_2b_3 - a_3b_2 \equiv 0 \mod p'
\end{aligned}
\end{equation}
\end{lemma}

\begin{theorem}
Let $p$ be an arbitrary prime, then $x_1^2 + x_2^2 + x_3^2 + x_4^2 = p$ is solvable.
\end{theorem}

\begin{proof}
Let $(a_1,a_2,a_3,a_4,a_5)$ be a reduced solution of equation (1) and $a_5 \in S$. Let $p' = p_n$ and $(b_1,b_2,b_3,b_4,b_5)$ be a reduced solution of equation (2) subject to (3). By taking the product of equations (1) and (2), we obtain:
\begin{equation}
(a_1^2 + a_2^2 + a_3^2 + a_4^2)(b_1^2 + b_2^2 + b_3^2 + b_4^2) - p p_n a_5b_5 = 0
\end{equation}
This equation by Euler's identity is really:
\begin{equation}
c_1^2 + c_2^2 + c_3^2 + c_4^2 - pp_na_5b_5 = 0
\end{equation}

\noindent Where by lemma 2, $\gcd(c_1,c_2,c_3,c_4) =d $ where $d \equiv 0 \mod p_n$. We can therefore reduce by $d$ and obtain:
\begin{equation}
\begin{cases}
{a_1^{1}}^2 + {a_2^{1}}^2  + {a_3^{1}}^2  + {a_4^{1}}^2  - pa_5^{1} = 0 \\
(a_1^1,a_2^1,a_3^1,a_4^1,a_5^1) = 1
\end{cases}
\end{equation}

\noindent where $a_{i}^1 = \frac{c_i}{d}$ where $i=1...4$ and $a_5^1 = \frac{p_na_5b_5}{d^2}$. Let us show that $a_5 \succ a_5^1$. $L(a_5b_5) = L(a_5)$ follows from the fact that all prime of divisors of $b_5$ preceed $p_n$. By multiplying $a_5b_5$ by $p_n$, $L(p_na_5b_5) = L(a_5)$. When one divides by $d^2$, there are two possibilities:
\begin{itemize}
\item Possibility 1: $L(\frac{p_na_5b_5}{d^2}) < L(a_5)$, then $ a_5^1 \prec a_5 $. \\
\item Possibility 2: $L(\frac{p_na_5b_5}{d^2}) = L(a_5)$, however in this case, $\nu(\frac{p_na_5b_5}{d^2}) < \nu(a_5)$, then $a_5^1 \prec a_5$.
\end{itemize}

\noindent If $L(a_5) > L(a_5^1)$, then the leading prime factor of $a_5^1$ strictly precedes $p_n$. Otherwise if $L(a_5) = L(a_5^1)$ (implying  $\nu(a_5) > \nu(a_5^1)$), we repeat this procedure for $p_n$. 
\newline
\newline
To finalize the proof, one should note that this reduction procedure can be repeated. Moreover, the maximal bound before $L$ turns into a strict inequality towards its predecessor is equal to the power of the prime we are reducing over. Therefore, we have the following ordered chain and its associated finite non-increasing sequence:
\begin{eqnarray}
a_5 \succ a_5^1 \succ ... \succ a_5^{k-1} \succ a_5^k \succ  ... \succ a_5^t = 1 \\
L(a_5) \geq L(a_5^1) \geq ... \geq L(a_5^{k-1})> L(a_5^k)\geq  ... > L(a_5^t) = 0
\end{eqnarray}

\noindent which completes the proof.

\end{proof}

\begin{theorem}[Lagrange's four square theorem]
By the unique factorization property of $\mathbb{Z}$, theorem 1, and Euler's identity, Lagrange's theorem follows.
\end{theorem}

\section{Remarks (Sketch of proofs of lemma 1 and lemma 2)}

Proof of Lemma 1: Using Chevalley's theorem, the equation $x_1^2 + x_2^2 + x_3^2 + x_4^2 - px_5 = 0$ has a non-zero solution $(a_1,a_2,a_3,a_4,a_5)$. It is possible find a solution $(b_1,b_2,b_3,b_4,b_5)$, such that every $b_i$,$(i=1...4)$ is a least absolute residue modulo $p$, consequently all prime of factors of $b_5$ preceed $p$. Let  $\gcd(b_1,b_2,b_3,b_4)=d$. If $d=1$, then we have a reduced solution. Otherwise, by dividing by $d$ we obtain the reduced solution.
\newline
\newline
\noindent Proof of Lemma 2: One can assume per the lemma's formulation that we are given $(a_1,a_2,a_3,a_4,a_5)$ which are a reduced solution of equation (1). Find the least absolute residues $(c_1,c_2,c_3,c_4)$ corresponding to $(a_1,a_2,a_3,a_4)$ modulo $p'$ (where $p'$ is an abitrary prime divisor of $a_5$). Using the properties of residues and their arithmetic, one can manipulate the $a$'s and $c$'s to obtain solutions over $p'$ without the relative-primality condition. This last condition is satisfied independently from $p'$ arithmetic, by dividing the final result by $\gcd(c_1,c_2,c_3,c_4)$, which results in $(b_1,b_2,b_3,b_4,b_5)$.
\newline
\newline

\bibliographystyle{ieeetr}
\bibliography{references}

\end{document}